\documentclass[twoside,reqno,centertags]{amsart}
\usepackage{color}
\usepackage{amsmath,amsthm,amsfonts,amssymb}
\usepackage[utf8]{inputenc}
\usepackage[english]{babel}
\usepackage{cite}
\usepackage[sort, numbers]{natbib}

 \usepackage{hyperref}
\hypersetup{colorlinks=true,    linkcolor=blue,    filecolor=blue,    urlcolor=blue,    citecolor=blue,}
\begin{document}
\title{\sc Multiplicity results for fractional Schr\"odinger-Kirchhoff systems involving critical nonlinearities}
\date{}
\maketitle
\numberwithin{equation}{section}
\allowdisplaybreaks
\newtheorem{theorem}{Theorem}[section]
\newtheorem{proposition}{Proposition}[section]
 \newtheorem{gindextheorem}{General Index Theorem}[section]
\newtheorem{indextheorem} {Index Theorem}[section]
\newtheorem{standardbasis}{Standard Basis}[section]
\newtheorem{generators} {Generators}[section]
\newtheorem{lemma} {Lemma}[section]
\newtheorem{corollary}{Corollary}[section]
\newtheorem{example}{Example}[section]
\newtheorem{examples}{Examples}[section]
\newtheorem{exercise}{Exercise}[section]
\newtheorem{remark}{Remark}[section]
\newtheorem{remarks}{Remarks}[section]
\newtheorem{definition} {{Definition}}[section]
\newtheorem{definitions}{Definitions}[section]
\newtheorem{notation}{Notation}[section]
\newtheorem{notations}{Notations}[section]
\newtheorem{defnot}{Definitions and Notations}[section]
\def\H{{\mathbb H}}
\def\N{{\mathbb N}}
\def\R{{\mathbb R}}
\newcommand{\be} {\begin{equation}}
\newcommand{\ee} {\end{equation}}
\newcommand{\bea} {\begin{eqnarray}}
\newcommand{\eea} {\end{eqnarray}}
\newcommand{\Bea} {\begin{eqnarray*}}
\newcommand{\Eea} {\end{eqnarray*}}
\newcommand{\p} {\partial}
\newcommand{\ov} {\over}
\newcommand{\al} {\alpha}
\newcommand{\ba} {\beta}
\newcommand{\de} {\delta}
\newcommand{\ga} {\gamma}
\newcommand{\Ga} {\Gamma}
\newcommand{\Om} {\Omega}
\newcommand{\om} {\omega}
\newcommand{\De} {\Delta}
\newcommand{\la} {\lambda}
\newcommand{\si} {\sigma}
\newcommand{\Si} {\Sigma}
\newcommand{\La} {\Lambda}
\newcommand{\no} {\nonumber}
\newcommand{\noi} {\noindent}
\newcommand{\lab} {\label}
\newcommand{\na} {\nabla}
\newcommand{\vp} {\varphi}
\newcommand{\var} {\varepsilon}
\newcommand{\RR}{{\mathbb R}}
\newcommand{\CC}{{\mathbb C}}
\newcommand{\NN}{{\mathbb N}}
\newcommand{\ZZ}{{\mathbb Z}}
\renewcommand{\SS}{{\mathbb S}}
\newcommand{\esssup}{\mathop{\rm {ess\,sup}}\limits}   
\newcommand{\essinf}{\mathop{\rm {ess\,inf}}\limits}   
\newcommand{\weaklim}{\mathop{\rm {weak-lim}}\limits}
\newcommand{\wstarlim}{\mathop{\rm {weak^\ast-lim}}\limits}
\newcommand{\RE}{\Re {\mathfrak e}}   
\newcommand{\IM}{\Im {\mathfrak m}}   
\renewcommand{\colon}{:\,}
\newcommand{\eps}{\varepsilon}
\newcommand{\half}{\textstyle\frac12}
\newcommand{\Takac}{Tak\'a\v{c}}
\newcommand{\eqdef}{\stackrel{{\rm {def}}}{=}}   
\newcommand{\wstarconverge}{\stackrel{*}{\rightharpoonup}}
\newcommand{\Div}{\nabla\cdot}
\newcommand{\Curl}{\nabla\times}
\newcommand{\Meas}{\mathop{\mathrm{meas}}}
\newcommand{\Int}{\mathop{\mathrm{Int}}}
\newcommand{\Clos}{\mathop{\mathrm{Clos}}}
\newcommand{\Lin}{\mathop{\mathrm{lin}}}
\newcommand{\Dist}{\mathop{\mathrm{dist}}}
\newcommand{\Square}{$\sqcap$\hskip -1.5ex $\sqcup$}
\newcommand{\Blacksquare}{\vrule height 1.7ex width 1.7ex depth 0.2ex }
\begin{center}
{\sc S. Fareh$^a$, K.  Akrout$^a$,  A.  Ghanmi$^b,$  D. D.  Repov\v{s}$^{c,d,e,}$\footnote{Corresponding author}}
\end{center}
\begin{abstract}
In this paper, we study certain critical Schr\"odinger-Kirchhoff type systems involving the fractional $p$-Laplace operator on a bounded domain. More precisely, using the properties of the associated functional energy on the  Nehari manifold sets and exploiting the analysis of the fibering map, we establish the multiplicity of solutions for such systems.\\

\noindent
{\sc  Keywords:} Variational method, Nehari manifold, elliptic equation, multiplicity of solutions.\\

\noindent
{\sc Math. Subj. Classif. (2020):} 35P30, 35J35, 35J60.
\end{abstract}

\section{Introduction}\label{sec1}
In recent years, 
a lot of
attention has been paid
 to  problems involving fractional and nonlocal operators. These types of problems arise in applications in many fields, e.g., in materials science \cite{9}, phase transitions \cite{5,38}, water
waves \cite{18, 19}, minimal surfaces \cite{15},
and
conservation laws \cite{10}.
For more applications of such problems in physical phenomena, probability, and finances, we refer interested readers to \cite{14, 16, 46}. Due to their importance, there are many interesting works on  existence and multiplicity of solutions for fractional and nonlocal problems either on bounded domains or on the entire space, see \cite{1, 3, 4, 6, 25, 26, PRR, 35, 36, 37}. 

In the last decade, many scholars have paid extensive attention to the Kirchhoff-type elliptic equations with critical exponents, see \cite{22, 27, 33}, for the  bounded domains and \cite{28, 30, 31} for the entire space. In particular, in  \cite{24}, the authors considered the following   Kirchhoff problem
\begin{equation} \label{vald}
\left\{
\begin{array}{l}
M(\underset{\mathbb{R}^{2n}}{\int\int} \frac{\left\vert u(x)-u(y)\right\vert^{2}}{ \vert x-y\vert ^{n+2s}}dxdy)(-\Delta)^{s}u=\lambda f(x,u)+ \left\vert u\right\vert
^{2^{*}_{s}-2}u\ \text{in }\Omega , \\
\\
u=0\text{ on}\ \mathbb{R}^{n}\setminus \Omega,%
\end{array}%
\right.
\end{equation}%
where $t\geq 0$  and $M(t)=a+bt$ for some  $a>0$ and $b\geq0$. Here, and in the rest of this paper,  $\Omega $ will denote a bounded domain in $\mathbb{R}^{n}$ with Lipschitz boundary $\partial\Omega$.  

Under suitable conditions, and by using the truncation technique method combined with the Mountain pass theorem,  the authors proved that for $\lambda>0$ large enough, problem \eqref{vald} has at least one nontrivial solution. Later, the fractional Kirchhoff-type problems were extensively studied by many authors using different methods, see \cite{7, 8, BRS,  17, 23, 29, 32, 11, 34, 39, 41, 42, 43, 44}.
In particular, by using the Nehari manifold method and the Symmetric mountain pass theorem,
 the authors in \cite{42} investigated the multiplicity of solutions for some $p$-Kirchhoff system with Dirichlet boundary conditions.

Mingqi et al. \cite{32}  studied the following Schr\"odinger–Kirchhoff type system
\begin{equation}  \label{Ef}
\left\{
\begin{array}{l}
M\mathbf{(}[(u,v)]^{p}_{s,p}+\left\Vert u,v\right\Vert^{p}_{p,V}\mathbf{)}\left(\mathcal{L}^{s}_{p}u+V(x)\left\vert u\right\vert^{p-2}u\right)=\lambda H_{u}(x,u,v)+\frac{\alpha}{p^{*}_{s}}\vert v\vert^{\beta}\vert u\vert^{\alpha-2}u  \ \text{in }\mathbb{R}^{n} , \\
M\mathbf{(}[(u,v)]^{p}_{s,p}+\left\Vert u,v\right\Vert^{p}_{p,V}\mathbf{)}\left(\mathcal{L}^{s}_{p}v+V(x)\left\vert v\right\vert^{p-2}v\right)=\lambda H_{v}(x,u,v)+\frac{\beta}{p^{*}_{s}}\vert u\vert^{\alpha}\vert v\vert^{\beta-2}v  \ \text{in }\mathbb{R}^{n},
\end{array}
\right.
\end{equation}
where $\lambda>0$, $\alpha+\beta=p^{*}_{s}:=\frac{np}{n-sp}$, $V:\mathbb{R}^{n}\rightarrow [0,\infty)$ is a continuous function, the Kirchhoff function $M:(0,\infty)\rightarrow(0,\infty)$ is continuous,
and
  $H_{u}$ and $H_{v}$ are  Caratheodory functions. Under some suitable assumptions and by applying the Mountain pass theorem with Ekeland’s variational principle, the authors obtained the existence and asymptotic behavior of
solutions for  system \eqref{Ef}. 

By the same methods as in \cite{32}, Fiscella et al. \cite{23} studied the existence of solutions for a critical Hardy-Schr\"odinger–Kirchhoff type system involving the fractional $p$-Laplacian in $\mathbb{R}^{n}$. Using the Three critical points theorem,
 Azroul et al.  \cite{8} established the existence of three weak solutions for a fractional $ p$-Kirchhoff-type system on a bounded domain with homogeneous Dirichlet boundary conditions.
Recently, Benkirane et al. \cite{7} have established the existence of three solutions for the $(p,q)$-Schr\"odinger-Kirchhoff type system in $\mathbb{R}^{n}$ via the Three critical points theorem.

Motivated by the above-mentioned papers,  we consider in this paper
the following Schr\"odinger-Kirchhoff type system involving the fractional $p$-Laplacian and critical nonlinearities
\begin{equation}  \label{E}
\left\{
\begin{array}{l}
M_{1}\mathbf{(}\left\Vert u\right\Vert^{p}_{V_{1}}\mathbf{)}\left((-\Delta)^{s}_{p}u+V_{1}(x)\left\vert u\right\vert^{p-2}u\right)=a_{1}(x) \left\vert u\right\vert
^{p^{*}_{s}-2}u+\lambda f(x,u,v)\ \text{in }\Omega , \\
M_{2}\mathbf{(}\left\Vert v\right\Vert^{p}_{V_{2}}\mathbf{)}\left((-\Delta)^{s}_{p}v+V_{2}(x)\left\vert v\right\vert^{p-2}v\right)=a_{2}(x) \left\vert v\right\vert
^{p^{*}_{s}-2}v+\lambda g(x,u,v)\ \text{in }\Omega , \\
u,v>0 \ \text{in }\Omega , \\
u=v=0\text{ on}\ \mathbb{R}^{n}\setminus \Omega,%
\end{array}
\right.
\end{equation}
where $\left\Vert .\right\Vert_{V_{1}}$ and  $\left\Vert .\right\Vert_{V_{2}}$  will be given later (see \eqref{norm}), $n>ps, 0<s<1<q<p$, $\lambda $ is a positive parameter, the weight functions $a_{1}$ and $a_{2}$ are positive and bounded on $\Omega,$
and
 $\mathbf{(-}\Delta \mathbf{)}_{p}^{s}$ is the fractional
$p$-Laplace operator, 
 defined by
\begin{equation*}
\left( -\Delta \right) _{p}^{s}u=2\underset{\varepsilon \rightarrow 0}{\lim }%
\int_{\mathbb{R}^{N}\backslash B_{\varepsilon }(x)}\frac{\left\vert
u(x)-u(y)\right\vert ^{p-2}(u(x)-u(y))}{\left\vert x-y\right\vert ^{n+ps}}%
dy, \
\hbox{for all}
\
x\in \mathbb{R}^{n},
\end{equation*}
where $B_{\varepsilon }(x)=\{y\in \mathbb{R}^{n}:\vert x-y\vert<\varepsilon\}.$
For more details about the fractional $p$-Laplacian operator and the basic properties of fractional Sobolev spaces, we refer the reader to \cite{20}.

 Throughout this paper, the index $i$ will denote integers $1$ or $2$, and  we
 shall
  assume that  the potential function $V_{i}:\Omega\rightarrow(0,\infty)$ is continuous and
  that 
   there exists $v_{i}>0$ such that $\underset{\Omega}{\inf}\,  V_{i}\geq v_{i}$.
 Also, we
 shall
  assume that $M_{i} :(0,\infty)\rightarrow(0,\infty),$ is a continuous function  satisfying the following conditions\newline
$(H_{1}) \underset{t\rightarrow\infty }{\lim }t^{1-\frac{p^{*}_{s}}{p}}M_{i}(t)=0$.\newline
$(H_{2})$ There exists $m_{i}>0$ such that for all $t>0$, we have 
$    M_{i}(t)\geq m_{i}.$\newline 
$(H_{3})$ There exists $\theta_{i} \in [1,\frac{p_{s}^{*}}{p}[$ such that for all $t>0$, we have
$
 M_{i}(t)t\leq \theta_{i} \widehat{M}_{i}(t),
$
where $\widehat{M}_{i}(t)=\int_{0}^{t}M_{i}(s)ds$.

Moreover, we shall assume
 that $f,g\in C(\bar\Omega\times \mathbb{R}\times\mathbb{R},[0,\infty[)$ are positively homogeneous
functions of degree $(q-1)$,
i.e., for all $t>0$ and $(x,u,v)\in \Omega \times \mathbb{R}\times \mathbb{R},$ we have
\begin{equation}  \label{homogen}
\left\{
\begin{array}{c}
f(x,tu,tv)=t^{q-1}f(x,u,v),\\
g(x,tu,tv)=t^{q-1}g(x,u,v).
\end{array}
\right.
\end{equation}%
Finally, we shall also assume that there exists a function $H:\bar{\Omega}\times\mathbb{R}\times\mathbb{R}\rightarrow \mathbb{R}$ satisfying
\begin{equation*}
    H_{u}(x,u,v)=f(x,u,v) \text{ and } H_{v}(x,u,v)=g(x,u,v),
\end{equation*}
where $H_{u}$ (respectively, $H_{v}$) denotes the partial derivative of $H$ with respect to $u$ (respectively, $v$). We note that the primitive function $H$ belongs to $C^{1}(\bar{\Omega}\times\mathbb{R}\times\mathbb{R},\mathbb{R})$ and  satisfies the following assumptions
for all $t>0, (x,u,v)\in \bar{\Omega} \times \mathbb{R}\times \mathbb{R},$
and some constant $ \gamma>0,$
\begin{equation}\label{H}
\left\{
\begin{array}{ll}
  H(x,tu,tv) = t^{q}H(x,u,v) , \\
  qH(x,u,v) = uf(x,u,v)+vg(x,u,v),\\
  \vert H(x,u,v)\vert \leq \gamma(\vert u\vert^{q}+\vert v\vert^{q}).
\end{array}
\right.
\end{equation}
Before stating our main  result, let us introduce some notations. For $s\in (0,1)$, we define the functional space
\begin{equation*}
W^{s,p}(Q)=\left\{ w:\mathbb{R}^{n}\to \mathbb{R}\ \text{measurable: }
w\in L^{p}(\Omega )\text{ and }\frac{w(x)-w(y)}{\left\vert x-y\right\vert ^{\frac{n}{p}+s}}\in L^{p}(Q)\right\} ,
\end{equation*}%
which is endowed with the norm
\begin{equation*}
\left\Vert w\right\Vert _{W^{s,p}(Q)}=\left( \left\Vert w\right\Vert _{L^{p}(\Omega
)}^{p}+\int\limits_{Q}\frac{\vert w(x)-w(y)\vert^{p}}{\left\vert x-y\right\vert ^{n+ps}}dxdy\right) ^{\frac{1}{p}},
\end{equation*}
where $Q=\mathbb{R}^{2n}\diagdown (\Omega ^{c}\times \Omega ^{c})$
 and $\Omega ^{c}=\mathbb{R}^{n}\setminus \Omega$.
From now on,  we shall  denote by $\Vert .\Vert_{q}$ the norm on the Lebesgue space ${L^{q}(\Omega)}$. It is well known that $\left(W^{s,p}(Q),\Vert .\Vert_{W^{s,p}(Q)}\right)$ is a uniformly convex Banach space.

Next, $ \ L^{p}(\Omega,V_{i})$ denotes the Lebesgue space  of real-valued functions,
with $V_{i}(x)\vert w\vert^{p}\in L^{1}(\Omega)$, endowed with the following norm
\begin{equation*}
    \Vert w\Vert_{p,V_{i}}=\left(\int\limits_{\Omega}V_{i}(x)\vert w\vert^{p}dx\right)^{\frac{1}{p}}.
\end{equation*}
Let us denote by $W^{s,p}_{V_{i}}(Q)$ the completion of $C_{0}^{\infty}(Q)$ with respect to the norm
\begin{equation}\label{norm}
    \left\Vert w\right\Vert_{V_{i}}=\left(\Vert w\Vert_{p,V_{i}}^{p}+\int\limits_{Q}\frac{\vert w(x)-w(y)\vert^{p}}{\left\vert x-y\right\vert ^{n+ps}}dxdy\right)^{\frac{1}{p}}.
\end{equation}
According to   \cite[(Theorem 6.7]{20}),  the embedding $W^{s,p}_{V_{i}}(Q)\hookrightarrow L^{\nu}(\Omega)$ is continuous for any $\nu\in [p,p^{\ast}_{s}]$. Namely, there exists a positive constant $C_{\nu}$ such that
\begin{equation*}
    \Vert w\Vert_{\nu}\leq C_{\nu}\Vert w\Vert_{V_{i}} \text{ for all } w\in W^{s,p}_{V_{i}}(Q).
\end{equation*}
Moreover, by   \cite[Lemma 2.1]{45}, the embedding  from  $W^{s,p}_{V_{i}}(Q)$ into $L^{\nu}(\Omega)$, is compact for any $\nu\in [1,p^{\ast}_{s})$.

Let $W=W^{s,p}_{V_{1}}(Q)\times W^{s,p}_{V_{2}}(Q)$ be equipped with the norm
$
    \Vert (u,v)\Vert=(\Vert u\Vert_{V_{1}}^{p}+\Vert v\Vert_{V_{2}}^{p})^{\frac{1}{p}}.
$
Then $(W,\Vert.\Vert)$ is a reflexive Banach space. The interested reader can refer to \cite{2} for more details.
Let $S_{p,V_{i}}$ be the best Sobolev constants for the embeddings from $W^{s,p}_{V_{i}}(Q)$ into $L^{p_{s}^{\ast }}(\Omega)$, which is given by
\begin{equation}  \label{sp}
S_{p,V_{i}}=\underset{u\in W^{s,p}_{V_{i}}(Q)\setminus\{0\}}{\inf }\frac{\left\Vert w\right\Vert_{V_{i}}^{p}}{\left\Vert w\right\Vert _{p_{s}^{\ast }}^{p}}.
\end{equation}
For simplicity, in the rest of this paper, $S$ will denote the following expression
\begin{equation}\label{SSS}
    S=\min(S_{p,V_{1}},S_{p,V_{2}}).
\end{equation}
Next, we define the notion of solutions for problem \eqref{E}.
\begin{definition}\label{def1.1}
We say that $(u,v)\in W$ is a weak solution of problem \eqref{E}, if
\begin{equation*}
\begin{array}{l}
M_{1}\left(\left\Vert u\right\Vert^{p}\right)\left(\int\limits_{Q}\frac{\left\vert u(x)-u(y)\right\vert^{p-2}(u(x)-u(y))(z(x)-z(y))}{\vert x-y\vert^{n+ps}}dxdy+\int\limits_{\Omega}V_{1}(x)\left\vert u\right\vert^{p-2}uzdx\right) \\
+M_{2}\left(\left\Vert v\right\Vert^{p}\right)\left(\int\limits_{Q}\frac{\left\vert v(x)-v(y)\right\vert^{p-2}(v(x)-v(y))(w(x)-w(y))}{\vert x-y\vert^{n+ps}}dxdy+\int\limits_{\Omega}V_{2}(x)\left\vert v\right\vert^{p-2}vwdx\right) \\
=\int\limits_{\Omega}(a_{1}(x) \left\vert u\right\vert^{p^{*}
_{s}-2}uz+a_{2}(x) \left\vert v\right\vert^{p^{*}
_{s}-2}vw)dx+\lambda\int\limits_{\Omega}(H_{u}(x,u,v)z+ H_{v}(x,u,v)w)dx,
\end{array}%
\end{equation*}
for all $(z,w)\in W.$
\end{definition}
The main result of this paper  is the following.
\begin{theorem}
\label{theo02} Assume  that $s\in (0,1),n>ps$, $1<q<p<p_{s}^{\ast }$, and
that  equations \eqref{homogen}, \eqref{H} hold. If  $M$ satisfies conditions $(H_{1})-(H_{3})$, then there exists $\lambda^{\ast }>0$ such that for all $\lambda \in (0,\lambda ^{\ast })$, system \eqref{E} has at least two nontrivial weak  solutions.
\end{theorem}

This paper is organized as follows. In Section $\ref{S02}$, we present some notations and preliminary results related to the Nehari manifold and fibering maps. In Section $\ref{S03}$, we prove Theorem $\ref{theo02}$.

\section{The Nehari manifold method and fibering maps analysis}
\label{S02}
This section collects some basic results on the Nehari manifold method and the fibering maps analysis which will be used in the
forthcoming section, we refer the interested reader to \cite{13, 14, 21}, for more details. We begin by  considering the Euler-Lagrange functional $J_{\lambda}: W\rightarrow \mathbb{R}$, which is defined by
\begin{equation}
J_{\lambda}(u,v)=\frac{1}{p}(\widehat{M}_{1}(A_{1}(u))+\widehat{M}_{2}(A_{2}(v)))-\frac{1}{p_{s}^{*}}B(u,v)-\lambda C(u,v),
\end{equation}
where
$$
  A_{i}\left( w\right) = \left\Vert w\right\Vert_{V_{i}}^{p}, 
  B\left( u,v\right) = \int\limits_{\Omega}(a_{1}(x) \left\vert u\right\vert^{p^{*}
_{s}}+a_{2}(x) \left\vert v\right\vert^{p^{*}
_{s}})dx,
C\left( u,v\right) =\int\limits_{\Omega }H(x,u,v)dx.
$$
We can easily verify that
 $J_{\lambda}\in C^{1}(W,\mathbb{R})$, moreover its derivative $J_{\lambda
}^{\prime }$ from the space  $W$ into its dual space $ W^{\prime }$ is given by
\begin{equation}\label{J'}
\langle J_{\lambda}^{\prime }(u,v),(u,v)\rangle= A_{1}(u){M}_{1}(A_{1}(u))+A_{2}(v){M}_{2}(A_{2}(v))-B(u,v)-\lambda qC(u,v).
\end{equation}
From the last equation, we can see that  the critical points of the functional $J_{\lambda}$  are exactly the weak solutions for problem \eqref{E}. Moreover, since the energy functional $J_{\lambda}$ is not bounded from below on $W,$  we shall  show that $J_{\lambda}$ is bounded from below on a suitable subset of $W,$ which is  known as the Nehari manifold and
is defined by
\begin{equation*}
\mathcal{N}_{\lambda}=\{(u,v)\in W\backslash \{(0,0)\},\langle J_{\lambda}^{\prime }(u,v),(u,v)\rangle _{W}=0\}.
\end{equation*}
It is clear that $(u,v)\in \mathcal{N}_{\lambda}$ if and only if
\begin{equation}\label{nlambd}
A_{1}(u){M}_{1}(A_{1}(u))+A_{2}(v){M}_{2}(A_{2}(v))-B(u,v)-\lambda q C(u,v)=0.
\end{equation}
Hence, from \eqref{J'}, we see that elements of $\mathcal{N}_{\lambda}$ correspond to nontrivial critical points which are solutions of problem \eqref{E}. 

It is useful to understand $\mathcal{N}_{\lambda}$ in terms of the
stationary points of the fibering maps $\varphi _{u,v}:(0,\infty) \rightarrow \mathbb{R}
$, defined by
$$
\varphi _{u,v}(t)= J_{\lambda}(tu,tv)\\=\frac{1}{p}(\widehat{M}_{1}(t^{p}A_{1}(u))+\widehat{M}_{2}(t^{p}A_{2}(v)))-\frac{t^{p_{s}^{*}}}{p_{s}^{*}}B(u,v)-\lambda t^{q}C(u,v).
$$
A simple calculation shows that for all $t>0$, we have
\begin{eqnarray*}
\varphi _{u,v}^{\prime }(t) &=& t^{p-1}(A_{1}(u)M_{1}(t^{p}A_{1}(u))+A_{2}(v)M_{2}(t^{p}A_{2}(v))) \\
&&-t^{p^{\ast}_{s}-1}B(u,v)-\lambda qt^{q-1}C(u,v),
\end{eqnarray*}
and
\begin{eqnarray*}
  \varphi _{u,v}^{\prime\prime }(t) &=& (p-1)t^{p-2}\left(A_{1}(u)M_{1}(t^{p}A_{1}(u))+A_{2}(v)M_{2}(t^{p}A_{2}(v))\right)\\
  &&+pt^{2p-2}\left((A_{1}(u))^{2}M^{\prime }_{1}(t^{p}A_{1}(u))+(A_{2}(v))^{2}M^{\prime }_{2}(t^{p}A_{2}(v)\right) \\
  &&-(p^{\ast}_{s}-1)t^{p^{\ast}_{s}-2}B(u,v)-\lambda q(q-1)t^{q-2}C(u,v).
\end{eqnarray*}
It is easy to see that for all $t>0$, we have
$$
\varphi _{u,v}^{\prime }(t)=\langle J_{\lambda}^{\prime }(tu,tv),(u,v)\rangle
_{W}\\=\frac{1}{t^{2}}\langle J_{\lambda  }^{\prime }(tu,tv),(tu,tv)\rangle
_{W}.
$$
So, $(tu,tv)\in \mathcal{N}_{\lambda}$
 if and only if $\varphi _{u,v}^{\prime
}(t)=0.$
In the special
case, when $t=1$, we get $(u,v)\in \mathcal{N}_{\lambda}$, if and only if $%
\varphi _{u,v}^{\prime }(1)=0.$ On the other hand, from \eqref{nlambd}, we obtain
\begin{eqnarray}
\varphi _{u,v}^{\prime \prime }(1) &=&(p-1)\left(
A_{1}(u)M_{1}(A_{1}(u))+A_{2}(v)M_{2}(A_{2}(v))\right) -(p_{s}^{\ast
}-1)B(u,v)  \notag \\
&&+p\left( (A_{1}(u))^{2}M_{1}^{\prime
}(A_{1}(u))+(A_{2}(v))^{2}M_{2}^{\prime }(A_{2}(v))\right) -\lambda
q(q-1)C(u,v)  \notag \\
&=&p\left( (A_{1}(u))^{2}M_{1}^{\prime
}(A_{1}(u))+(A_{2}(v))^{2}M_{2}^{\prime }(A_{2}(v))\right)   \notag \\
&&-(p_{s}^{\ast }-p)B(u,v)-\lambda q(q-p)C(u,v)  \label{22} \\
&=&p\left( (A_{1}(u))^{2}M_{1}^{\prime
}(A_{1}(u))+(A_{2}(v))^{2}M_{2}^{\prime }(A_{2}(v))\right) +\lambda
q(p_{s}^{\ast }-q)C(u,v)  \notag \\
&&-(p_{s}^{\ast }-p)\left(
A_{1}(u)M_{1}(A_{1}(u))+A_{2}(v)M_{2}(A_{2}(v))\right)   \label{c} \\
&=&p\left( (A_{1}(u))^{2}M_{1}^{\prime
}(A_{1}(u))+(A_{2}(v))^{2}M_{2}^{\prime }(A_{2}(v))\right) -(p_{s}^{\ast
}-q)B(u,v)  \notag \\
&&+(p-q)\left( A_{1}(u)M_{1}(A_{1}(u))+A_{2}(v)M_{2}(A_{2}(v))\right) .
\label{bb}
\end{eqnarray}
Now, in order to obtain a multiplicity of solutions, we divide $\mathcal{N}_{\lambda}$ into three parts
$$ \mathcal{N}_{\lambda }^{+} =\left\{ (u,v)\in \mathcal{N}_{\lambda
}:\varphi _{u,v}^{\prime \prime }(1)>0\right\} =\left\{ (u,v)\in W:\varphi
_{u,v}^{\prime}(1)=0 \;\mbox{and}\;\varphi _{u,v}^{\prime \prime }(1)>0\right\},
$$
$$ 
\mathcal{N}_{\lambda }^{-} =\left\{ (u,v)\in \mathcal{N}_{\lambda
}:\varphi _{u,v}^{\prime \prime }(1)<0\right\}=\left\{ (u,v)\in W:\varphi
_{u,v}^{\prime}(1)=0 \;\mbox{and}\;\varphi _{u,v}^{\prime \prime }(1)<0\right\} ,
$$
$$ 
\mathcal{N}_{\lambda }^{0} =\left\{ (u,v)\in \mathcal{N}_{\lambda
}:\varphi _{u,v}^{\prime \prime }(1)=0\right\}=\left\{ (u,v)\in W:\varphi
_{u,v}^{\prime}(1)=0 \;\mbox{and}\;\varphi _{u,v}^{\prime \prime }(1)=0\right\} .
$$
\begin{lemma}
\label{lem01} Suppose that $(u_{0},v_{0})$ is a local minimizer for $J_{\lambda}$ on $%
\mathcal{N}_{\lambda}$, with $(u_{0},v_{0})\not\in \mathcal{N}%
_{\lambda}^0.$
Then $(u_{0},v_{0})$ is a critical point of $J_{\lambda}.$
\end{lemma}

\begin{proof}
If $(u_{0},v_{0})$ is a local minimizer for $J_{\lambda}$ on $\mathcal{N}%
_{\lambda}$, then $(u_{0},v_{0})$ solves the 
following 
optimization problem
\begin{equation*}
\left\{
\begin{array}{l}
\displaystyle\min_{(u,v)\in \mathcal{N}_{\lambda}} J_{\lambda
}(u,v)=J_{\lambda}(u_{0},v_{0}), \\
\beta (u_{0},v_{0})=0,%
\end{array}%
\right.
\end{equation*}%
where
\begin{equation*}
\beta (u,v)=A_{1}(u){M}_{1}(A_{1}(u))+A_{2}(v){M}_{2}(A_{2}(v))-B(u,v)-\lambda q C(u,v).
\end{equation*}%
By the Lagrangian multipliers theorem, there exists $\delta \in \mathbb{R}$,
such that
\begin{equation}  \label{jprime}
J_{\lambda}^{\prime }(u_{0},v_{0})=\delta \beta ^{\prime }(u_{0},v_{0}).
\end{equation}%
Since $(u_{0},v_{0})\in \mathcal{N}_{\lambda }$, we obtain
\begin{equation}  \label{betaprime1}
\delta \langle \beta^{\prime }(u_{0},v_{0}),(u_{0},v_{0})\rangle _{W}=\langle
J_{\lambda }^{\prime }(u_{0},v_{0}),(u_{0},v_{0})\rangle _{W}=0.
\end{equation}%
Moreover, by \eqref{nlambd} and the constraint $\beta (u_{0},v_{0})=0$,
we have
\begin{eqnarray*}
  \langle \beta ^{\prime }(u_{0},v_{0}),(u_{0},v_{0})\rangle_{W}&=&p\left((A_{1}(u_{0}))^{2}M_{1}^{\prime}(A_{1}(u_{0}))+(A_{2}(v_{0}))^{2}M_{2}^{\prime }(A_{2}(v_{0}))\right) \\
  &&-(p^{*}_{s}-p)B(u_{0},v_{0})-\lambda q(q-p)C(u_{0},v_{0}) 
  =\varphi_{u_{0},v_{0}}^{\prime \prime }(1).
\end{eqnarray*}
Since $(u_{0},v_{0})\not\in\mathcal{N}_{\lambda}^{0}$, we have $\varphi_{u_{0},v_{0}}^{\prime \prime }(1)\neq0$. Thus, by \eqref{betaprime1} we get $\delta =0$.
Consequently, by substitution of $\delta$ in \eqref{jprime}, we obtain $
J_{\lambda }^{\prime }(u_{0},v_{0})=0$. This completes the proof of Lemma \ref{lem01}.
\end{proof}
In order to understand the Nehari manifold and fibering maps, let us
define the function $\psi _{u,v}:(0,\infty) \rightarrow \mathbb{R} $
as
follows
\begin{equation}\label{psiti}
\psi _{u,v}(t)=t^{p-q}\left(A_{1}(u)M_{1}(t^{p}A_{1}(u))+A_{2}(v)M_{2}(t^{p}A_{2}(v))\right)-t^{p^{*}_{s}-q}B(u,v)-\lambda qC(u,v).
\end{equation}
We note that
$
   t^{q-1}\psi _{u,v}(t)=\varphi^{\prime}_{u,v}(t).
$
Thus, it is easy to see that
 $(tu,tv)\in \mathcal{N}_{\lambda }$ if and only if
\begin{equation}  \label{psic}
\psi _{u,v}(t)=0.
\end{equation}
Moreover, by a direct computation, we get
\begin{eqnarray*}
\psi _{u,v}^{\prime }(t) &=&(p-q)t^{p-q-1}\left(
A_{1}(u)M_{1}(t^{p}A_{1}(u))+A_{2}(v)M_{2}(t^{p}A_{2}(v))\right)  \\
&&+pt^{2p-q-1}\left( A_{1}^{2}(u)M_{1}^{\prime
}(t^{p}A_{1}(u))+A_{2}^{2}(v)M_{2}^{\prime }(t^{p}A_{2}(v))\right)  \\
&&-(p_{s}^{\ast }-q)t^{p_{s}^{\ast }-q-1}B(u,v).
\end{eqnarray*}
Therefore
\begin{equation}\label{phiprime}
t^{q-1}\psi _{u,v}^{\prime }(t)=\varphi _{u,v}^{\prime \prime }(t).
\end{equation}
Hence, $(tu,tv)\in \mathcal{N}_{\lambda }^{+}$
 $\left (\text{respectively, } (tu,tv) \in \mathcal{N}_{\lambda  }^{-}\right)$ if and only if $\psi _{u,v}(t)=0$ and $\psi_{u,v}^{\prime }(t)>0$, (respectively, $\psi _{u,v}(t)=0$, and $\psi _{u,v}^{\prime}(t)<
0$). 
Put \begin{equation}\label{mm}
    m=\min(m_{1},m_{2}), \;\; \theta=\max(\theta_{1},\theta_{2}),
\end{equation}
and
\begin{equation}\label{lamda}
    \lambda_{\ast}=\frac{(mS)^{\frac{p_{s}^{\ast }-q}{p_{s}^{\ast }-p}}}{\gamma q\vert\Omega\vert^{\frac{p_{s}^{\ast }-q}{p_{s}^{\ast }}}}\left(\frac{p_{s}^{\ast }-p}{p_{s}^{\ast }-q}\right)\left(\frac{p-q}{(p_{s}^{\ast }-q)a}\right)^{\frac{p-q}{p_{s}^{\ast }-p}}.
\end{equation}
Now we shall prove the following crucial result.
\begin{lemma}
\label{lem03} Assume that  conditions $(H_{1})$ and $(H_{2})$ hold. Then for all $(u,v)\in \mathcal{N}_{\lambda}$, there exist $\lambda_{\ast}>0$ and unique $t_{1}>0$ and $t_{2}>0$, such that for each $\lambda \in (0,\lambda_{\ast})$, we have $(t_{1}u,t_{1}v)\in \mathcal{N}_{\lambda }^{+}$ and $(t_{2}u,t_{2}v)\in \mathcal{N}_{\lambda }^{-}$.
\end{lemma}
\begin{proof}
We begin by noting that by \eqref{psiti}, we have  $$\psi _{u,v}(t)\rightarrow -\lambda qC(u,v),\;\mbox{ as }t\rightarrow 0^{+}, \;\mbox{and}\;\psi _{u,v}(t)\rightarrow -\infty,\;\mbox{as }t\rightarrow \infty.$$ 
Now, if we combine equations  \eqref{H} and \eqref{sp} with  the H\"{o}lder inequality, we obtain
\begin{eqnarray}  \label{42}
B(u,v)&\leq &\Vert a_{1}\Vert_{\infty}\Vert u\Vert _{p_{s}^{\ast }}^{p_{s}^{\ast }}+\Vert a_{2}\Vert_{\infty}\Vert v\Vert _{p_{s}^{\ast }}^{p_{s}^{\ast }} 
\leq  a(\Vert u\Vert _{p_{s}^{\ast }}^{p_{s}^{\ast }}+\Vert v\Vert _{p_{s}^{\ast }}^{p_{s}^{\ast }}) \nonumber\\
&\leq & a(S_{p,V_{1}}^{-\frac{p_{s}^{\ast }}{p}}(A_{1}(u))^{\frac{p_{s}^{\ast }}{p}}+S_{p,V_{2}}^{-\frac{p_{s}^{\ast }}{p}}(A_{2}(u))^{\frac{p_{s}^{\ast }}{p}}) \nonumber\\
&\leq & S^{-\frac{p_{s}^{\ast }}{p}}a(A(u,v))^{\frac{p_{s}^{\ast }}{p}},
\end{eqnarray}
and
\begin{eqnarray}  \label{43}
C(u,v)&\leq& \gamma (\left\Vert u\right\Vert _{q}^{q}+\left\Vert v\right\Vert _{q}^{q})
\leq \gamma\left\vert \Omega \right\vert ^{\frac{p_{s}^{\ast }-q}{p_{s}^{\ast }}
}(\left\Vert u\right\Vert _{p_{s}^{\ast }}^{q}+\left\Vert v\right\Vert _{p_{s}^{\ast }}^{q})\nonumber\\
&\leq& \gamma S^{-\frac{q}{p}}\left\vert \Omega \right\vert ^{\frac{p_{s}^{\ast }-q}{p_{s}^{\ast }}}(A(u,v))^{\frac{q}{p}},
\end{eqnarray}
where $a=\max(\Vert a_{1}\Vert_{\infty},\Vert a_{2}\Vert_{\infty})$,  $A(u,v)=\Vert (u,v)\Vert^{p},$  and $S$ is given by equation \eqref{SSS}.

On the other hand, by combining equations  \eqref{42}, \eqref{43} with $(H_{2})$, we obtain
\begin{eqnarray}
\psi_{u,v} \left( t\right)&\geq&t^{p-q}(m_{1}A_{1}(u)+m_{2}A_{2}(v))-t^{p^{\ast}_{s}-q}S^{-\frac{p_{s}^{\ast }}{p}}a(A(u,v))^{\frac{p_{s}^{\ast }}{p}}\notag \\
&&-\lambda q \gamma S^{-\frac{q}{p}}\left\vert \Omega \right\vert ^{
\frac{p_{s}^{\ast }-q}{p_{s}^{\ast }}}(A(u,v))^{\frac{q}{p}} \notag \\
 &\geq& mt^{p-q}A(u,v)-t^{p^{\ast}_{s}-q}S^{-\frac{p_{s}^{\ast }}{p}}a(A(u,v))^{\frac{p_{s}^{\ast }}{p}}\notag\\
&&-\lambda q \gamma S^{-\frac{q}{p}}\left\vert \Omega \right\vert ^{
\frac{p_{s}^{\ast }-q}{p_{s}^{\ast }}}(A(u,v))^{\frac{q}{p}}  
\geq (A(u,v))^{\frac{q}{p}}F_{u,v}(t),\label{fiu}
\end{eqnarray}%
where  $m$ is given by equation  \eqref{mm} and $F_{u,v}$ is defined for $t>0$ by
$$
  F_{u,v}(t) = mt^{p-q}(A(u,v))^{\frac{p-q}{p}}-t^{p^{\ast}_{s}-q}S^{-\frac{p_{s}^{\ast }}{p}}a(A(u,v))^{\frac{p_{s}^{\ast }-q}{p}}  -\lambda q \gamma S^{-\frac{q}{p}}\left\vert \Omega \right\vert ^{
\frac{p_{s}^{\ast }-q}{p_{s}^{\ast }}}.
$$
Since $1<q<p<p^{\ast}_{s}$, 
it is easy to see that $\underset{t\rightarrow 0^{+} }{\lim }F_{u,v}(t)<0$
and 
$\underset{t\rightarrow \infty }{\lim }F_{u,v}(t)=-\infty.$
So, by a simple calculation we can prove that $F_{u,v}$ attains its unique global maximum at
\begin{equation}
    t_{max}(u,v)=\left(\frac{m}{S^{-\frac{p_{s}^{\ast }}{p}}a}\left(\frac{p-q}{p^{\ast}_{s}-q}\right)\right)^{\frac{1}{p^{\ast}_{s}-p}}(A(u,v))^{\frac{-1}{p}}.
\end{equation}
Moreover,
\begin{equation}
    F_{u,v}(t_{max})=q\gamma S^{-\frac{q}{p}}\left\vert \Omega \right\vert ^{
\frac{p_{s}^{\ast }-q}{p_{s}^{\ast }}}(\lambda_{\ast}-\lambda),
\end{equation}
where $\lambda_{\ast}$ is given by \eqref{lamda}.

If we choose $\lambda<\lambda_{\ast}$, then  we get from $\eqref{fiu}$
\begin{equation}
    \psi_{u,v}(t_{max})\geq(A(u,v))^{\frac{q}{p}}F_{u,v}(t_{max})>0.
\end{equation}
Hence, 
by a
 variation of $\psi_{u,v}(t)$ there exist unique $t_{1}<t_{max}(u,v)$ and unique $t_{2}>t_{max}(u,v)$, such that $\psi^{\prime}_{u,v}(t_{1})>0$ and $\psi^{\prime}_{u,v}(t_{2})<0$. Moreover
$
   \psi_{u,v}(t_{1})=0=\psi_{u,v}(t_{2}).
$
Finally, it follows from (\ref{psic}) and (\ref{phiprime}) that $
(t_{1}u,t_{1}v)\in \mathcal{N}_{\lambda }^{+}$ and $(t_{2}u,t_{2}v)\in \mathcal{N}
_{\lambda }^{-}$.
This completes the proof of Lemma~\ref{lem03}.
\end{proof}
We can see from Lemma \ref{lem03}
 that sets $ \mathcal{N}_{\lambda }^{+}$ and $ \mathcal{N}_{\lambda }^{-}$ are nonempty. In the following lemma we shall
 give a property related to $ \mathcal{N}_{\lambda }^{0}$.
\begin{lemma}
\label{lem04} Assume that condition $(H_{2})$ holds. Then  for all $\lambda \in \left( 0,\lambda_{\ast } \right)$, we have $ \mathcal{N}_{\lambda }^{0}=\emptyset $.
\end{lemma}
\begin{proof}
We shall argue by contradiction. Assume that there exists $\lambda>0$ in $(0,\lambda_{\ast})$ such that  $\mathcal{N}_{\lambda}^0\neq \emptyset$. Let  $(u_{0},v_{0})\in \mathcal{N}_{\lambda}^0.$ 
Then invoking $(H_{2})$,  $\eqref{c}$ and $\eqref{43}$, we have
\begin{eqnarray}
  0=\varphi''_{u}(1)&=&p\left((A_{1}(u))^{2}M_{1}^{\prime}(A_{1}(u))+(A_{2}(v))^{2}M_{2}^{\prime }(A_{2}(v))\right) \notag\\
  &&-(p^{*}_{s}-p)\left(A_{1}(u)M_{1}(A_{1}(u))+A_{2}(v)M_{2}(A_{2}(v))\right)+\lambda q(p^{*}_{s}-q)C(u,v)\notag\\
  &\leq&p\left((A_{1}(u))^{2}M_{1}^{\prime}(A_{1}(u))+(A_{2}(v))^{2}M_{2}^{\prime }(A_{2}(v))\right)\notag\\
  &&-(p^{*}_{s}-p)\left(m_{1}A_{1}(u)+m_{2}A_{2}(v)\right)+\lambda q(p^{*}_{s}-q)C(u,v)\notag\\
   &\leq& p\left((A(u))^{2}M^{\prime}(A(u))+(A(v))^{2}N^{\prime }(A(v))\right)-(p^{*}_{s}-p)m A(u,v)\notag\\
   &&+\lambda q(p^{*}_{s}-q)\gamma S^{-\frac{q}{p}}\vert\Omega\vert^{\frac{p_{s}^{\ast}-q}{p_{s}^{\ast}}}(A(u,v))^{\frac{q}{p}}.\label{fi}
\end{eqnarray}
On the other hand, by
 $(H_{2})$, $\eqref{bb}$ and $\eqref{42}$, one has
\begin{eqnarray}
  0=\varphi''_{u}(1)&=&p\left((A_{1}(u))^{2}M_{1}^{\prime}(A_{1}(u))+(A_{2}(v))^{2}M_{2}^{\prime }(A_{2}(v))\right)\notag\\
  &&+(p-q)\left(A_{1}(u)M_{1}(A_{1}(u))+A_{2}(v)M_{2}(A_{2}(v))\right)-(p^{*}_{s}-q)B(u,v) \notag\\
  &\geq&p\left((A_{1}(u))^{2}M_{1}^{\prime}(A_{1}(u))+(A_{2}(v))^{2}M_{2}^{\prime }(A_{2}(v))\right)\notag\\
  &&+(p-q)\left(m_{1}A_{1}(u)+m_{2}A_{2}(v)\right)-(p^{*}_{s}-q)B(u,v) \notag\\
  &\geq& p\left((A_{1}(u))^{2}M_{1}^{\prime}(A_{1}(u))+(A_{2}(v))^{2}M_{2}^{\prime }(A_{2}(v))\right)\notag\\
  &&+(p-q)m A(u,v)-(p^{*}_{s}-q)S^{-\frac{p^{\ast}_{s}}{p}}a (A(u,v))^{\frac{p_{s}^{\ast}}{p}}.\label{fii}
\end{eqnarray}
Combining $\eqref{fi}$ and $\eqref{fii}$, we get
\begin{equation}\label{da}
    \lambda \geq\frac{m(A(u,v))^{\frac{p-q}{p}}-S^{-\frac{p^{\ast}_{s}}{p}}a(A(u,v))^{\frac{p^{\ast}_{s}-q}{p}}}{q\gamma S^{-\frac{q}{p}}\vert\Omega\vert^{\frac{p_{s}^{\ast}-q}{p_{s}^{\ast}}}}.
\end{equation}
Next, we define the function $H$ on $(0,\infty)$ by
\begin{equation*}
   H(t)=\frac{mt^{\frac{p-q}{p}}-S^{-\frac{p^{\ast}_{s}}{p}}
   a t^{\frac{p^{\ast}_{s}-q}{p}}}{q\gamma S^{-\frac{q}{p}}\vert\Omega\vert^{\frac{p_{s}^{\ast}-q}{p_{s}^{\ast}}}}.
\end{equation*}
 Since $1<q<p<p^{\ast}_{s}$,
 it follows that
  $\underset{t\rightarrow0^{+}}{\lim}H(t)=0$ and $\underset{t\rightarrow\infty}{\lim}H(t)=-\infty$. A simple computation 
  now
  shows that $H$ attains its maximum at
\begin{equation*}
    \tilde{t}=\left(\left(\frac{p-q}{p^{\ast}_{s}-q}\right)\frac{mS^{\frac{p^{\ast}_{s}}{p}}}
    {a}\right)^{\frac{p}{p^{\ast}_{s}-p}},
\end{equation*}
and
\begin{equation}\label{daa}
\underset{t>0}{\max}\; H(t)=H(\tilde{t})=\lambda_{\ast}.
\end{equation}
Hence it follows from $\eqref{da}$ and $\eqref{daa}$, 
$
    \lambda\geq\underset{t>0}{\max} \;H(t)=\lambda_{\ast},
$
which contradicts $\lambda\in(0,\lambda_{\ast})$. Therefore we can conclude that that indeed
$\mathcal{N}_{\lambda }^{0}=\emptyset $, for $\lambda\in(0,\lambda_{\ast})$.
This completes the proof of Lemma~\ref{lem04}.
\end{proof}
\begin{lemma}
\label{lem05} Assume that  conditions $(H_{2})$ and $(H_{3})$ hold. Then $J_{\lambda }$ is coercive and bounded from below on$\
\mathcal{N}_{\lambda }$.
\end{lemma}
\begin{proof}
Let $(u,v)\in \mathcal{N}_{\lambda }$. Then  by $\eqref{nlambd} $, we get
\begin{equation*}
B(u,v)=A_{1}(u)M_{1}(A_{1}(u))+A_{2}(v)M_{2}(A_{2}(v))-\lambda qC(u,v).
\end{equation*}
Therefore
\begin{eqnarray*}
  J_{\lambda }(u,v) &=& \frac{1}{p}(\widehat{M}_{1}(A_{1}(u))+\widehat{M}_{2}(A_{2}(v)))-\frac{1}{p^{\ast}_{s}}(A_{1}(u)M_{1}(A_{1}(u)) \\
  && +A_{2}(v)M_{2}(A_{2}(v))) -\lambda \left(1- \frac{q}{p^{\ast}_{s}}\right)C(u,v).
\end{eqnarray*}
Moreover, by   $(H_{2}),(H_{3})$ and  \eqref{43}, we have
\begin{eqnarray*}
  J_{\lambda }(u,v)&\geq&\frac{1}{\theta_{1} p}A_{1}(u)M_{1}(A_{1}(u))+\frac{1}{\theta_{2} p}A_{2}(v)M_{2}(A_{2}(v))-\frac{1}{p^{\ast}_{s}}A_{1}(u)M_{1}(A_{1}(u)) \\
  && -\frac{1}{p^{\ast}_{s}}A_{2}(v)M_{2}(A_{2}(v))-\lambda \left(1- \frac{q}{p^{\ast}_{s}}\right)C(u,v)\\
  &\geq&\left(\frac{1}{\theta p}-\frac{1}{p^{\ast}_{s}}\right)(A_{1}(u)M_{1}(A_{1}(u))+A_{2}(v)M_{2}(A_{2}(v)))-\lambda \left(1- \frac{q}{p^{\ast}_{s}}\right)C(u,v) \\
  &\geq&\left(\frac{1}{\theta p}-\frac{1}{p^{\ast}_{s}}\right)(m_{1}A_{1}(u)+m_{2}A_{2}(v))-\lambda \left(1- \frac{q}{p^{\ast}_{s}}\right)C(u,v) \\
  &\geq& m\left(\frac{1}{\theta p}-\frac{1}{p^{\ast}_{s}}\right)A(u,v)-\lambda \left(1- \frac{q}{p^{\ast}_{s}}\right)\gamma S^{-\frac{q}{p}}\left\vert \Omega \right\vert ^{
\frac{p_{s}^{\ast }-q}{p_{s}^{\ast }}}(A(u,v))^{\frac{q}{p}}.
\end{eqnarray*}
Since $q<p$ and $\theta p<p^{\ast}_{s}$, it follows that  $J_{\lambda }$ is coercive and bounded from below on $\mathcal{N}_{\lambda }$.
This completes the proof of Lemma~\ref{lem05}.
\end{proof}
By Lemma (\ref{lem04}), we can write $\mathcal{N}_{\lambda }=
\mathcal{N}_{\lambda }^{+}\cup \mathcal{N}_{\lambda }^{-}$, and by
Lemma (\ref{lem05}), we can define
\begin{equation*}
\alpha _{\lambda }^{-}=\underset{(u,v)\in \mathcal{N}_{\lambda  }^{-}}{
\inf }J_{\lambda }(u,v)\text{ and }\alpha _{\lambda  }^{+}=\underset{
(u,v)\in \mathcal{N}_{\lambda }^{+}}{\inf }J_{\lambda }(u,v).
\end{equation*}
\section{Proof of the main result}
\label{S03}
In this section,   we shall prove the main result of this paper (Theorem \ref{theo02}). First, we need to prove two propositions.
\begin{proposition} \label{T0}
 Assume that  conditions $(H_{2})$ and $(H_{3})$ hold. Then there exist $t_0>0$ and  $(u_{0},v_{0})\in W\backslash \{0\}$,  with $(u_{0},v_{0})>0$ in$\
\mathbb{R}^{n}$, such that
\begin{equation}  \label{**}
\frac{1}{p}(\widehat{M}_{1}(A_{1}(u_{0})t_0^{p})+\widehat{M_{2}}(A_{2}({v_{0}})t_{0}^{p}))-\frac{t_0^{p_{s}^{\ast }}}{p^{\ast}_{s}}B(u_{0},v_{0})=\left(\frac{s}{n}-\frac{\theta-1}{\theta p}\right)a ^{\frac{-n}{sp_{s}^{\ast }}}(\frac{mS}{\theta})^{\frac{n}{s p}}.
\end{equation}
\end{proposition}
\begin{proof}
For any $(u,v)\in W\setminus\{0\}$, we define the function $\zeta_{u,v}: (0,\infty) \to \mathbb{R}$, as
follows
\begin{eqnarray*}
 \zeta_{u,v} \left( t\right) &=& \frac{1}{p}(\widehat{M}_{1}(A_{1}({tu}))+\widehat{M_{2}}(A_{2}({tv})))-\frac{1}{p^{\ast}_{s}}B({t(u,v)}) \\
   &=& \frac{1}{p}(\widehat{M}_{1}(t^{p}A_{1}({u}))+\widehat{M}_{2}(t^{p}A_{2}({v})))
-\frac{t^{p^{\ast}_{s}}}{p^{\ast}_{s}}B({u,v}).
\end{eqnarray*}
By $(H_{3}),$ it can be shown that $\underset{t\rightarrow0^{+}}{\lim}\zeta_{u,v}(t)\geq 0$ and $\underset{t\rightarrow\infty}{\lim}\zeta_{u,v}(t)=-\infty$.
It is clear that $\zeta$ is of class $C^1.$
Moreover, 
invoking
 $(H_{2})$ and $(H_{3})$, we obtain
\begin{eqnarray*}
  \zeta_{u,v} \left( t\right)&\geq&\frac{t^{p}}{\theta_{1} p}A_{1}(u)M_{1}(t^{p}A_{1}(u))+\frac{t^{p}}{\theta_{2} p}A_{2}(v)M_{2}(t^{p}A_{2}(v))-\frac{t^{p^{\ast}_{s}}}{p^{\ast}_{s}}B({u,v}) \\
   &\geq& \frac{t^{p}}{\theta p}(A_{1}(u)M_{1}(t^{p}A_{1}(u))+A_{2}(v)M_{2}(t^{p}A_{2}(v)))-\frac{t^{p^{\ast}_{s}}}{p^{\ast}_{s}}B({u,v}) \\
   &\geq&\frac{t^{p}}{\theta p}(m_{1}A_{1}(u)+m_{2}A_{2}(v))-\frac{t^{p^{\ast}_{s}}}{p^{\ast}_{s}}B({u,v}) \\
   &\geq& \frac{m}{\theta p}t^{p}A(u,v)-\frac{t^{p^{\ast}_{s}}}{p^{\ast}_{s}}B({u,v})=\omega_{u,v}(t).
\end{eqnarray*}
Since $\displaystyle\lim_{t\to 0}\omega_{u,v}(t)=0$ and $\displaystyle\lim_{t\to \infty}\omega_{u,v}(t)=-\infty$, it follows that  $\omega_{u,v}$ attains its global maximum at
\begin{equation*}
t_{\ast }=\left( \frac{mA(u,v)}{\theta B(u,v)}\right) ^{\frac{1}{p_{s}^{\ast }-p}}.
\end{equation*}%
Moreover, from $\eqref{42}$ and the fact that $p^{\ast}_{s}>\theta p$, we have
\begin{eqnarray}
  \underset{t>0}{\sup }\;\omega_{u,v} \left( t\right)&=&\omega_{u,v}(t_{\ast })\notag \\
  &=& \left(\frac{p^{\ast}_{s}-p}{pp^{\ast}_{s}}\right)\left(\frac{m}{\theta}\right)^
  {\frac{p^{\ast}_{s}}{p^{\ast}_{s}-p}}(A(u,v))^{\frac{p^{\ast}_{s}}{p^{\ast}_{s}-p}}(B(u,v))^{-\frac{p}{p^{\ast}_{s}-p}}\notag \\
  &=& \left(\frac{p^{\ast}_{s}-p}{pp^{\ast}_{s}}\right)\left(\frac{m}{\theta}\right)^
  {\frac{p^{\ast}_{s}}{p^{\ast}_{s}-p}}\left((A(u,v))^{-\frac{p^{\ast}_{s}}{p}}B(u,v)\right)^{-\frac{p}{p^{\ast}_{s}-p}} \notag \\
  &=& \frac{s}{n}\left(\frac{m}{\theta}\right)^{\frac{n}{s p}}(A(u,v))^{\frac{n}{s p}}(B(u,v))^{-\frac{n}{sp^{\ast}_{s}}}\notag \\
  &\geq& \frac{s}{n}a ^{\frac{-n}{sp_{s}^{\ast }}}(\frac{m S}{\theta})^{\frac{n}{s p}}\notag \\
  &\geq&\left(\frac{s}{n}-\frac{\theta-1}{\theta p}\right)a ^{\frac{-n}{sp_{s}^{\ast }}}(\frac{m S}{\theta})^{\frac{n}{s p}}\label{ll} \\
  &=&\left(\frac{p_{s}^{\ast }-\theta p}{\theta pp_{s}^{\ast }}\right)a ^{\frac{-n}{sp_{s}^{\ast }}}(\frac{m S}{\theta})^{\frac{n}{s p}}>0.\notag
\end{eqnarray}
Therefore, using the variations of the functions $\zeta_{u,v}$ and $\omega_{u,v}$, we get
\begin{equation*}
    \underset{t>0}{\sup }\;\zeta_{u,v}\geq\underset{t>0}{\sup }\; \omega_{u,v}\geq \left(\frac{s}{n}-\frac{\theta-1}{\theta p}\right)a ^{\frac{-n}{sp_{s}^{\ast }}}(\frac{mS}{\theta})^{\frac{n}{s p}}.
\end{equation*}
Hence, there exists ${t}_{0}>0,$ such that
\begin{equation*}
\zeta_{u,v} \left( {t}_{0}\right) =\left(\frac{s}{n}-\frac{\theta-1}{\theta p}\right)a ^{\frac{-n}{sp_{s}^{\ast }}}(\frac{mS}{\theta})^{\frac{n}{s p}}.
\end{equation*}
This completes the proof of Proposition 3.1.\end{proof}
Set now
\begin{equation}
L=\left(p-q\right) \left( \frac{m}{q}\left(\frac{s}{n}-\frac{\theta-1}{\theta p}\right)\right)
^{-\frac{q}{p-q}}\left( \frac{p_{s}^{\ast }-q}{\theta p^{2}}\right) ^{\frac{p}{p-q}}\left( \gamma S^{-\frac{q}{p}}\left\vert \Omega \right\vert
^{\frac{p_{s}^{\ast }-q}{p_{s}^{\ast }}}\right) ^{\frac{p}{p-q}}.
\label{m}
\end{equation}
\begin{proposition}
\label{prop01}  Assume that  conditions $(H_{2})$ and $(H_{3})$ hold. If $1<q<p<p_{s}^{\ast },$ then  every Palais-Smale sequence $\left\{ (u_{k},v_{k})\right\} $ $\subset W$ for $J_{\lambda
}$ at level $c$, with
\begin{equation}\label{L}
c<\left(\frac{s}{n}-\frac{\theta-1}{\theta p}\right) a ^{\frac{-n}{sp_{s}^{\ast}}}\left(\frac{mS}{\theta }\right)^{\frac{n}{sp}}-\lambda ^{\frac{p}{p-q}}L,
\end{equation}
possesses a convergent subsequence.
\end{proposition}
\begin{proof}
Let $\left\{ (u_{k},v_{k})\right\} $  be a Palais-Smale sequence for  $J_{\lambda
}$ at level $c$, that is $$J_{\lambda}(u_{k},v_{k})\to c, \;\mbox{and}\;J'_{\lambda}(u_{k},v_{k})\to 0, \;\mbox{as}\;k\to \infty.$$ 
By Lemma (\ref{lem05}), we know that $\left\{ (u_{k},v_{k})\right\} $ is bounded in
$W$. So up to a subsequence, still denoted by $\left\{ (u_{k},v_{k})\right\}$,
there exists $(u_{\ast },v_{\ast})\in W$, $\mu>0$ and $\eta>0$, such that as $k$ tends to infinity, we have
\begin{equation}\label{rz} 
\left\{
  \begin{array}{ll}
(u_{k},v_{k}) \rightharpoonup (u_{\ast },v_{\ast})\quad \text{weakly in } W,\\
\Vert u_{k}\Vert_{V_{1}}\rightarrow \mu,  \Vert v_{k}\Vert_{V_{2}}\rightarrow \eta, \\
  (u_{k},v_{k}) \rightharpoonup (u_{\ast },v_{\ast}) \quad \text{weakly in } L^{p_{s}^{\ast }}\left(
\Omega\right)\times L^{p_{s}^{\ast }}\left(
\Omega\right),\\
  (u_{k},v_{k}) \rightarrow (u_{\ast },v_{\ast}) \ \text{strongly in\ }L^{q}(
\Omega)\times L^{q}(
\Omega),\ 1\leq q<p_{s}^{\ast },\\
  (u_{k},v_{k}) \rightarrow (u_{\ast },v_{\ast})\ \text{a.e. in \ }\Omega,
  \end{array}
\right.
\end{equation}
Since $1\leq q<p_{s}^{\ast },$
it follows by \cite[Theorem IV-9]{40}, that there
exist functions  $l_{1},l_{2}\in L^{q}(\Omega)$ such that for a.e. $x\in \Omega$, we have 
$
\left\vert u_{k}(x)\right\vert \leq l_{1}(x),\;\; \left\vert v_{k}(x)\right\vert \leq l_{2}(x).
$
Hence, by the Dominated convergence theorem, 
\begin{equation}\label{hz}
C(u_{k},v_{k})\longrightarrow C(u_{\ast },v_{\ast}),\text{ as}\ k\rightarrow \infty .
\end{equation}
On the other hand, by the Brezis-Lieb lemma \cite [Lemma 1.32]{23}, for $k$ large enough, we have
\begin{equation*}
A_{1}(u_{k})=A_{1}(u_{k}-u_{\ast })+A_{1}(u_{\ast })+o(1),
\end{equation*}
\begin{equation*}
A_{2}(v_{k})=A_{2}(v_{k}-v_{\ast })+A_{2}(v_{\ast })+o(1),
\end{equation*}
and
\begin{equation*}
B(u_{k},v_{k})=B\left( u_{k}-u_{\ast },v_{k}-v_{\ast}\right) +B(u_{\ast },v_{\ast})+o(1).
\end{equation*}
Consequently, by letting $k$ tend to infinity, we get
\begin{eqnarray*}
o(1) &=&\langle J_{\lambda }^{\prime }(u_{k},v_{k}),(u_{k}-u_{\ast
},v_{k}-v_{\ast })\rangle _{W} \\
&=&M_{1}(A_{1}(u_{k}))\left( \int\limits_{Q}\frac{\left\vert
u_{k}(x)-u_{k}(y)\right\vert ^{p-1}((u_{k}-u_{\ast })(x)-(u_{k}-u_{\ast
})(y))}{|x-y|^{n+ps}}dxdy\right.  \\
&&\left. +\int\limits_{\Omega }V_{1}(x)\left\vert u_{k}\right\vert
^{p-1}(u_{k}-u_{\ast })dx\right) -\int\limits_{\Omega
}a_{1}(x)|u_{k}|^{p_{s}^{\ast }-1}(u_{k}-u_{\ast })dx \\
&&+M_{2}(A_{2}(v_{k}))\left( \int\limits_{Q}\frac{\left\vert
v_{k}(x)-v_{k}(y)\right\vert ^{p-1}((v_{k}-v_{\ast })(x)-(v_{k}-v_{\ast
})(y))}{|x-y|^{n+ps}}dxdy\right.  \\
&&\left. +\int\limits_{\Omega }V_{2}(x)\left\vert v_{k}\right\vert
^{p-1}(v_{k}-v_{\ast })dx\right) -\int\limits_{\Omega
}a_{2}(x)|v_{k}|^{p_{s}^{\ast }-1}(v_{k}-v_{\ast })dx \\
&&-\lambda \int\limits_{\Omega }(H_{u}(x,u_{k},v_{k})(u_{k}-u_{\ast
})+H_{v}(x,u_{k},v_{k})(v_{k}-v_{\ast }))dx \\
&=&M_{1}(\mu ^{p})(\mu ^{p}-A_{1}(u_{\ast }))+M_{2}(\eta ^{p})(\eta
^{p}-A_{2}(v_{\ast })) \\
&&-\int\limits_{\Omega }(a_{1}(x)|u_{k}|^{p_{s}^{\ast
}}+a_{2}(x)|v_{k}|^{p_{s}^{\ast }})dx+\int\limits_{\Omega }(a_{1}(x)|u_{\ast
}|^{p_{s}^{\ast }}+a_{2}(x)|v_{\ast }|^{p_{s}^{\ast }})dx \\
&&-\lambda \int\limits_{\Omega }(H_{u}(x,u_{k},v_{k})(u_{k}-u_{\ast
})+H_{v}(x,u_{k},v_{k})(v_{k}-v_{\ast }))dx+o(1) \\
&=&M_{1}\left( \mu ^{p}\right) A_{1}\left( u_{k}-u_{\ast }\right)
+M_{2}\left( \eta ^{p}\right) A_{2}\left( v_{k}-v_{\ast }\right)
-B(u_{k}-u_{\ast },v_{k}-v_{\ast }) \\
&&-\lambda \int\limits_{\Omega }(H_{u}(x,u_{k},v_{k})(u_{k}-u_{\ast
})+H_{v}(x,u_{k},v_{k})(v_{k}-v_{\ast }))dx+o(1).
\end{eqnarray*}%
Therefore
\begin{equation*}
\begin{array}{l}
M_{1}\left( \mu^{p}\right)\underset{k\rightarrow\infty}{\lim}A_{1}(u_{k}-u_{\ast })+M_{2}\left( \eta^{p}\right)\underset{k\rightarrow\infty}{\lim}A_{2}(v_{k}-v_{\ast })= \underset{k\rightarrow\infty}{\lim}B\left( u_{k}-u_{\ast },v_{k}-v_{\ast }\right)\\
+\underset{k\rightarrow\infty}{\lim}\lambda\int\limits_{\Omega}(H_{u}(x,u_{k},v_{k})(u_{k}-u_{\ast})+H_{v}(x,u_{k},v_{k})(v_{k}-v_{\ast}))dx
\end{array}
\end{equation*}
By \eqref{H}, \eqref{rz} and the Holder inequality, it follows that
\begin{eqnarray*}
  \int\limits_{\Omega}(H_{u}(x,u_{k},v_{k})(u_{k}-u_{\ast})&+&H_{v}(x,u_{k},v_{k})(v_{k}-v_{\ast}))dx\\
   &\leq& \gamma q\int\limits_{\Omega}\vert u_{k}\vert^{q-1}(u_{k}-u_{\ast})dx+\gamma q\int\limits_{\Omega}\vert v_{k}\vert^{q-1}(v_{k}-v_{\ast})dx \\
  &\leq& \gamma q\Vert u_{k}\Vert_{q}^{q-1}\Vert u_{k}-u_{\ast}\Vert_{q}+\gamma q\Vert v_{k}\Vert_{q}^{q-1}\Vert v_{k}-v_{\ast}\Vert_{q}\\
  &\leq&C_{q}\gamma q\Vert u_{k}\Vert_{V_{1}}^{q-1}\Vert u_{k}-u_{\ast}\Vert_{q}+C_{q}\gamma q\Vert v_{k}\Vert_{V_{2}}^{q-1}\Vert v_{k}-v_{\ast}\Vert_{q},
\end{eqnarray*}
for some positive constant $C_{q}$.
So, we obtain
\begin{equation}\label{Hldr}
    \underset{k\rightarrow\infty}{\lim}\int\limits_{\Omega}(H_{u}(x,u_{k},v_{k})(u_{k}-u_{\ast})+H_{v}(x,u_{k},v_{k})(v_{k}-v_{\ast}))dx=0.
\end{equation}
Thus, from \eqref{Hldr}, we can deduce that
\begin{equation*}
\underset{k\rightarrow\infty}{\lim}B\left( u_{k}-u_{\ast },v_{k}-v_{\ast }\right)= M_{1}\left( \mu^{p}\right)\underset{k\rightarrow\infty}{\lim}A_{1}(u_{k}-u_{\ast })+M_{2}\left( \eta^{p}\right)\underset{k\rightarrow\infty}{\lim}A_{2}(v_{k}-v_{\ast }).
\end{equation*}

For simplicity, set
 $
b:=\displaystyle\lim_{k\rightarrow \infty
}B\left( u_{k}-u_{\ast }, v_{k}-v_{\ast }\right).
$
Note that  $b\geq 0.$ Moreover, to  prove that $(u_{k},v_{k})$ converges strongly to  $(u_{\ast},v_{\ast})$, it suffices to prove that $b=0$.
 Suppose to the contrary, that $b>0$.
Then  by $(H_{2})$, we get
\begin{eqnarray}
  A_{1}(u_{k}-u_{\ast })M_{1}(\mu^{p})+A_{2}(v_{k}-v_{\ast })M_{2}(\eta^{p}) &\geq& m_{1}A_{1}(u_{k}-u_{\ast })+m_{2}A_{2}(v_{k}-v_{\ast })\notag \\
  &\geq& m A(u_{k}-u_{\ast },v_{k}-v_{\ast }) .\label{sl}
\end{eqnarray}
Using \eqref{42}, we get
\begin{equation}\label{Bk}
    A(u_{k}-u_{\ast},v_{k}-v_{\ast })\geq Sa^{-\frac{p}{p_{s}^{\ast }}}(B\left( u_{k}-u_{\ast },v_{k}-v_{\ast }\right))^{\frac{p}{p_{s}^{\ast }}}.
\end{equation}
So by combining  \eqref{sl} with \eqref{Bk}, we obtain
\begin{equation*}
    A_{1}(u_{k}-u_{\ast })M_{1}(A_{1}(u_{\ast}))+A_{2}(v_{k}-v_{\ast })M_{2}(A_{2}(v_{\ast}))\geq mSa^{-\frac{p}{p_{s}^{\ast }}}(B\left( u_{k}-u_{\ast },v_{k}-v_{\ast }\right))^{\frac{p}{p_{s}^{\ast }}}.
\end{equation*}
By letting $k$ tend to infinity, we conclude that
\begin{equation}
b\geq a ^{\frac{-n}{sp_{s}^{\ast }}
}(mS)^{^{\frac{n}{sp}}}.  \label{b}
\end{equation}
On the other hand, by
 $(H_{3})$, \eqref{hz} and $\eqref{b}$, one has
\begin{eqnarray*}
c&=&\underset{k\longrightarrow \infty }{\lim }J_{\lambda  }(u_{k},v_{k})
 = \underset{k\longrightarrow \infty }{\lim }\left( J_{\lambda }(u_{k},v_{k})- \frac{1}{p_{s}^{\ast }}\langle J_{\lambda }^{\prime }(u_{k},v_{k}),(u_{k},v_{k})\rangle_{W}\right)  \\
&=&\underset{k\longrightarrow \infty }{\lim } \left[\frac{1}{p}\left(\widehat{M}_{1}\left( A_{1}(u_{k})\right)+\widehat{M}_{2}\left( A_{2}(v_{k})\right)\right)-\frac{1}{p_{s}^{\ast }}A_{1}(u_{k})M_{1}(A_{1}(u_{k}))\right.\\
&&\left.-\frac{1}{p_{s}^{\ast }} A_{2}(v_{k})M_{2}\left(A_{2}(v_{k})\right)-\lambda \left(\frac{p_{s}^{\ast }-q}{p_{s}^{\ast }}\right) C(u_{k},v_{k})\right]  \\
&\geq&\underset{k\longrightarrow \infty }{\lim }\left[\frac{1}{\theta_{1} p}A_{1}(u_{k})M_{1}(A_{1}(u_{k}))+\frac{1}{\theta_{2} p} A_{2}(v_{k})M_{2}\left(A_{2}(v_{k})\right)\right.\\
&&\left.-\frac{1}{p_{s}^{\ast }}\left(A_{1}(u_{k})M_{1}(A_{1}(u_{k}))+ A_{2}(v_{k})M_{2}\left(A_{2}(v_{k})\right)\right)-\lambda \left(\frac{p_{s}^{\ast }-q}{p_{s}^{\ast }}\right)C(u_{k},v_{k}) \right]\\
&\geq&\underset{k\longrightarrow \infty }{\lim }\left[\left(\frac{1}{\theta p}-\frac{1}{p_{s}^{\ast }}\right)\left(A_{1}(u_{k})M_{1}(A_{1}(u_{k}))+ A_{2}(v_{k})M_{2}\left(A_{2}(v_{k})\right)\right)\right.\\
&&\left. -\lambda \left(\frac{p_{s}^{\ast }-q}{p_{s}^{\ast }}\right)C(u_{k},v_{k}) \right]\\
&=& \underset{k\longrightarrow \infty }{\lim }\left[ \left(\frac{p_{s}^{\ast }-\theta p}{\theta pp_{s}^{\ast }}\right)A_{1}(u_{k})M_{1}(\mu^{p})+\left(\frac{p_{s}^{\ast }-\theta p}{\theta pp_{s}^{\ast }}\right)A_{2}(v_{k})M_{2}(\eta^{p})\right.\\
&&\left.-\lambda \left(\frac{p_{s}^{\ast }-q}{p_{s}^{\ast }}\right)C(u_{k},v_{k})\right] \\
&=& \underset{k\longrightarrow \infty }{\lim }\left[ \left(\frac{p_{s}^{\ast }-\theta p}{\theta pp_{s}^{\ast }}\right)\left(A_{1}(u_{k}-u_{\ast})M_{1}(\mu^{p})+A_{2}(v_{k}-v_{\ast})M_{2}(\eta^{p})\right)\right.\\
&&\left.+\left(\frac{p_{s}^{\ast }-\theta p}{\theta pp_{s}^{\ast }}\right)\left(A_{1}(u_{\ast})M_{1}(\mu^{p})+A_{2}(v_{\ast})M_{2}(\eta^{p})\right)-\lambda \left(\frac{p_{s}^{\ast }-q}{p_{s}^{\ast }}\right)C(u_{k},v_{k})\right]\\
&=& \left(\frac{s}{n}-\frac{\theta-1}{\theta p}\right)b+\left(\frac{s}{n}-\frac{\theta-1}{\theta p}\right)\left(A_{1}(u_{\ast})M_{1}(\mu^{p})+A_{2}(v_{\ast})M_{2}(\eta^{p})\right)\\
&&-\lambda \left(\frac{p_{s}^{\ast }-q}{p_{s}^{\ast }}\right)C(u_{\ast},v_{\ast})\\
&\geq& \left(\frac{s}{n}-\frac{\theta-1}{\theta p}\right)b+\left(\frac{s}{n}-\frac{\theta-1}{\theta p}\right)\left(m_{1}A_{1}(u_{\ast})+m_{2}A_{2}(v_{\ast})\right)
-\lambda \left(\frac{p_{s}^{\ast }-q}{p_{s}^{\ast }}\right)C(u_{\ast},v_{\ast})\\
&\geq&\left(\frac{s}{n}-\frac{\theta-1}{\theta p}\right)a ^{\frac{-n}{sp_{s}^{\ast }}
}(mS)^{^{\frac{n}{sp}}}+\left(\frac{s}{n}-\frac{\theta-1}{\theta p}\right)mA(u_{\ast},v_{\ast})
-\lambda \left(\frac{p_{s}^{\ast }-q}{p_{s}^{\ast }}\right)C(u_{\ast},v_{\ast}).
\end{eqnarray*}
Now,  from  \eqref{43}, and using the fact that $\theta p<p_{s}^{\ast }$, we obtain
\begin{eqnarray}
c&\geq&\left(\frac{s}{n}-\frac{\theta-1}{\theta p}\right)a ^{\frac{-n}{sp_{s}^{\ast }}
}(mS)^{^{\frac{n}{sp}}}+\left(\frac{s}{n}-\frac{\theta-1}{\theta p}\right)mA(u_{\ast },v_{\ast })  \notag \\
&&-\lambda \gamma S^{-\frac{q}{p}}\left\vert \Omega \right\vert ^{
\frac{p_{s}^{\ast }-q}{p_{s}^{\ast }}}\left( \frac{p_{s}^{\ast }-q}{\theta p}
\right)\left( A(u_{\ast },v_{\ast })\right) ^{\frac{q}{p}}  \notag \\
&=&\left(\frac{s}{n}-\frac{\theta-1}{\theta p}\right)a ^{\frac{-n}{sp_{s}^{\ast }}
}(mS)^{^{\frac{n}{sp}}}+h\left( A(u_{\ast },v_{\ast })\right) ,
\label{000}
\end{eqnarray}
where $h$ is defined on $[0,\infty)$ by
\begin{equation*}
h(\xi )=\left(\frac{s}{n}-\frac{\theta-1}{\theta p}\right)m\xi-\lambda \gamma S^{-\frac{q}{p}}\left\vert \Omega \right\vert ^{
\frac{p_{s}^{\ast }-q}{p_{s}^{\ast }}}\left( \frac{p_{s}^{\ast }-q}{\theta p}
\right)\xi ^{\frac{q}{p}}.
\end{equation*}%
A simple computation shows that $h$ attains its minimum at
\begin{equation*}
\xi _{0}=\left( \lambda q \gamma S^{-\frac{q}{p}}\left\vert \Omega \right\vert ^{
\frac{p_{s}^{\ast }-q}{p_{s}^{\ast }}}\left( \frac{p^{\ast}-q}{mp}\right)\frac{1}{\frac{s}{n}\theta p-(\theta-1)}\right) ^{\frac{p}{p-q}},
\end{equation*}
and%
\begin{equation}
\underset{\xi >0}{\inf }\;h\left( \xi \right) =h(\xi _{0})=-\lambda ^{\frac{p}{p-q}}L,
\label{001}
\end{equation}%
where $L$ is given by (\ref{m}).

Therefore, from \eqref{000}, \eqref{001}, and by considering $\theta\geq1$, we obtain
\begin{eqnarray*}
  c &\geq& \left(\frac{s}{n}-\frac{\theta-1}{\theta p}\right)a ^{\frac{-n}{sp_{s}^{\ast }}
}(mS)^{\frac{n}{sp}}-\lambda ^{\frac{p}{p-q}}L \\
  &\geq& \left(\frac{s}{n}-\frac{\theta-1}{\theta p}\right)a ^{\frac{-n}{sp_{s}^{\ast }}
}(\frac{mS}{\theta})^{\frac{n}{sp}}-\lambda ^{\frac{p}{p-q}}L.
\end{eqnarray*}
This contradicts \eqref{L}. Hence, $b=0$. So, we deduce that $(u_{k},v_{k})\rightarrow (u_{\ast },v_{\ast})$
strongly in $W$. This completes the proof.
\end{proof}
\begin{proposition}
\label{prop02}  Assume that  conditions $(H_{2})$ and $(H_{3})$ hold.
Then there exist $\lambda^{\ast }>0,$ ${t}_{0}>0$ and $
({u}_{0},v_{0})\in W$, such that
\begin{equation}
J_{\lambda  }({t}_{0}{u}_{0},{t}_{0}{v}_{0})\leq \left(\frac{s}{n}-\frac{\theta-1}{\theta p}\right)a ^{\frac{-n}{sp_{s}^{\ast }}
}(\frac{mS}{\theta})^{\frac{n}{sp}}-\lambda ^{\frac{p}{p-q}}L, \label{48}
\end{equation}%
provided that  $\lambda \in (0,\lambda^{\ast })$. In particular,
\begin{equation}
\alpha _{\lambda }^{-}<\left(\frac{s}{n}-\frac{\theta-1}{\theta p}\right)a ^{\frac{-n}{sp_{s}^{\ast }}
}(\frac{mS}{\theta})^{\frac{n}{sp}}-\lambda ^{\frac{p}{p-q}}L.
\end{equation}
\end{proposition}
\begin{proof}
We put
\begin{equation*}
\lambda_{\ast\ast}=\left(\frac{1}{L}\left(\frac{s}{n}-\frac{\theta-1}{\theta p}\right)a ^{\frac{-n}{sp_{s}^{\ast }}
}(\frac{mS}{\theta})^{\frac{n}{sp}}\right)^{\frac{p-q}{p}}.
\end{equation*}
Then  for any $0<\lambda<\lambda_{\ast\ast}$, we have
\begin{equation}
\left(\frac{s}{n}-\frac{\theta-1}{\theta p}\right)a ^{\frac{-n}{sp_{s}^{\ast }}
}(\frac{mS}{\theta})^{\frac{n}{sp}}-\lambda ^{\frac{p}{p-q}}L>0.  \label{111}
\end{equation}
By \eqref{**}, there exist ${t}_{0}>0$ and $({u}_{0},v_{0})\in W\setminus
\{0\}$, such that
\begin{eqnarray}
J_{\lambda }({t}_{0}{u}_{0},t_{0}v_{0}) &=&\frac{1}{p}\left(\widehat{M}_{1}(t_{0}^{p}A_{1}({u}_{0}))+\widehat{M}_{2}(t_{0}^{p}A_{2}({v}_{0}))\right)
-\frac{{t}_{0}^{p^{\ast}}}{p^{\ast}}B({u}_{0},v_{0})-\lambda{t}_{0}^{q}C(
{u}_{0},v_{0})    \notag \\
&=&\left(\frac{s}{n}-\frac{\theta-1}{\theta p}\right)a ^{\frac{-n}{sp_{s}^{\ast }}
}(\frac{mS}{\theta})^{\frac{n}{s p}}-\lambda{t}_{0}^{q}C({u}_{0},v_{0})  \label{411}
\end{eqnarray}
Let
\begin{equation*}
\lambda_{\ast\ast\ast}=\left( \frac{{t}_{0}^{q}C({u}_{0},v_{0})}{L},
\right) ^{\frac{p-q}{q}}.
\end{equation*}%
Then  for all $\lambda \in (0,\lambda_{\ast\ast\ast})$, we have
\begin{equation}
-\lambda{t}_{0}^{q}C({u}_{0},v_{0})<-\lambda ^{\frac{p}{p-q}}L.\label{dd}
\end{equation}%
Thus, from $\eqref{411}$ and $\eqref{dd}$, we get
\begin{equation*}
J_{\lambda }({t}_{0}{u}_{0},t_{0}v_{0})<\left(\frac{s}{n}-\frac{\theta-1}{\theta p}\right)a ^{\frac{-n}{sp_{s}^{\ast }}
}(\frac{mS}{\theta})^{\frac{n}{s p}}-\lambda ^{\frac{p}{p-q}}L.
\end{equation*}%
Hence, \eqref{48} holds. 
Finally, if we put
$
\lambda ^{\ast }=\min (\lambda _{\ast},\lambda _{\ast\ast},\lambda _{\ast\ast\ast}),
$
then  for all $0<\lambda <\lambda ^{\ast }$ and using the analysis of the fibering maps $\varphi _{u,v}(t)=J_{\lambda  }(tu,tv)$, we get
\begin{equation*}
\alpha _{\lambda }^{-}<\left(\frac{s}{n}-\frac{\theta-1}{\theta p}\right)a ^{\frac{-n}{sp_{s}^{\ast }}
}(\frac{mS}{\theta})^{\frac{n}{s p}}-\lambda ^{\frac{p}{p-q}}L.
\end{equation*}%
This completes the proof of Proposition~\ref{prop01}.
\end{proof}
Now, we are in a position to prove the main result of this paper.\\

\textit{ Proof of Theorem \ref{theo02}:} By Lemma \ref{lem05}, $J_{\lambda }$ is bounded from  below on $\mathcal{N}_{\lambda }$. Consequently, it is bounded from  below on  $\mathcal{N}_{\lambda }^{+}$ and  $\mathcal{N}_{\lambda }^{-}$. So, we can find    sequences $\{(u_{k}^{+},v_{k}^{+})\}\subset \mathcal{N}_{\lambda }^{+}$ and $\{(u_{k}^{-},v_{k}^{-})\}\subset \mathcal{N}_{\lambda }^{-}$, such that if  $k$ tends to infinity,  then
\begin{equation*}
J_{\lambda }(u_{k}^{+},v_{k}^{+}) \longrightarrow \inf_{(u,v)\in \mathcal{N}_{\lambda }^{+}}J_{\lambda }(u,v)=\alpha _{\lambda}^{+},
\end{equation*}
and
\begin{equation*}
J_{\lambda  }(u_{k}^{-},v_{k}^{-}) \longrightarrow \inf_{(u,v)\in \mathcal{N}_{\lambda }^{-}}J_{\lambda }(u,v)=\alpha _{\lambda
 }^{-}.
\end{equation*}
By an analysis of fibering maps $\varphi _{u,v}$ we can conclude that $\alpha _{\lambda }^{+}<0$ and $\alpha _{\lambda }^{-}>0$. Moreover, by Propositions \ref{prop01} and \ref{prop02}, we have
\begin{equation*}
J_{\lambda }(u_{k}^{+},v_{k}^{+}) \longrightarrow J_{\lambda }(u_{\ast}^{+},v_{\ast}^{+})=\inf_{(u,v)\in \mathcal{N}_{\lambda }^{+}}J_{\lambda }(u,v)=\alpha _{\lambda
}^{+},J_{\lambda }^{\prime }(u_{k}^{+},\;v_{k}^{+})\longrightarrow 0,
\end{equation*}
and
\begin{equation*}
J_{\lambda }(u_{k}^{-},v_{k}^{-}) \longrightarrow J_{\lambda }(u_{\ast}^{-},v_{\ast}^{-})=\inf_{(u,v)\in \mathcal{N}_{\lambda }^{-}}J_{\lambda }(u,v)=\alpha _{\lambda
}^{-}, \;J_{\lambda }^{\prime }(u_{k}^{-},v_{k}^{-})\longrightarrow 0.
\end{equation*}
Therefore, $(u_{\ast}^{+},v_{\ast}^{+})$ (respectively, $(u_{\ast}^{-},v_{\ast}^{-})$) is a minimizer of $J_{\lambda }$ on $\mathcal{N}_{\lambda }^{+}$ (respectively, on $\mathcal{N}_{\lambda }^{-})$.
Hence, by Lemma \ref{lem01}, problem (\ref{E}) has two solutions $(u_{\ast}^{+},v_{\ast}^{+})\in
\mathcal{N}_{\lambda  }^{+}\ $ and $(u_{\ast}^{-},v_{\ast}^{-})\in \mathcal{N}
_{\lambda }^{-}$.  Moreover, since $\mathcal{N}_{\lambda }^{+}\cap
\mathcal{N}_{\lambda }^{-}= \emptyset ,$ 
it follows that these two solutions are
distinct. 
Finally, the fact that $\alpha _{\lambda }^{+}<0$ and $\alpha _{\lambda }^{-}>0$ imply that  $(u_{\ast}^{+},v_{\ast}^{+})$  and $(u_{\ast}^{-},v_{\ast}^{-})$ are  nontrivial solutions for  problem \eqref{E}. 
This completes the proof of Theorem~\ref{theo02}.
\qed
\\

\subsection*{Acknowledgements.} 
The fourth author was supported by the Slovenian Research Agency program P1-0292 and
grants N1-0278, N1-0114, N1-0083, J1-4031, and J1-4001.

\vfill

$$$$
$^a$LAMIS Laboratory, Echahid Cheikh Larbi Tebessi University, 12002 Tebessa, Algeria. \\
Emails souraya.fareh@univ-tebessa.dz, kamel.akrout@univ-tebessa.dz
$$$$
$^b$LR10ES09 Mod\'elisation mat\'ematique, analyse harmonique et
t\'eorie du potentiel, Facult\'e des Sciences, Universit\'e de Tunis El Manar, 2092 Tunis, Tunisie.\\
Email: abdeljabbar.ghanmi@lamsin.rnu.tn
$$$$
$^c$Department of Mathematics and Computer Science, Faculty of Education, University of Ljubljana, 1000 Ljubljana, Slovenia. \\
Email: dusan.repovs@pef.uni-lj.si
$$$$
$^d$Department of Mathematics, Faculty of Mathematics and Physics, University of Ljubljana, 1000 Ljubljana, Slovenia. \\
Email: dusan.repovs@fmf.uni-lj.si
$$$$
$^e$Department of Mathematics, Institute of Mathematics, Physics and Mechanics, 1000 Ljubljana, Slovenia  \\
Email: dusan.repovs@guest.arnes.si
$$$$
\end{document}